\documentclass{amsart}
\usepackage{graphicx, color}
\usepackage{amscd}
\usepackage{amsmath}
\usepackage{amsfonts}
\usepackage{amssymb}
\usepackage{mathrsfs}
\newtheorem{theorem}{Theorem}[section]

\newtheorem{proposition}[theorem]{Proposition}
\newtheorem{remark}[theorem]{Remark}

\allowdisplaybreaks
\numberwithin{equation}{section}

\numberwithin{equation}{section}
\newcommand{\RR}{\mathbb R}
\newcommand{\NN}{\mathbb N}

\renewcommand{\leq}{\leqslant}

\renewcommand{\geq}{\geqslant}

\begin{document}

\title[Nonlinear problems on the Sierpi\'nski gasket]{Nonlinear problems on the Sierpi\'nski gasket}

\thanks{The first and the third author were supported by the INdAM-GNAMPA Project 2016 {\it Problemi variazionali su variet\`a riemanniane e gruppi di Carnot}, by the DiSBeF Research Project 2015 {\it Fenomeni non-locali: modelli e applicazioni}, by the DiSPeA Research Project 2016 {\it Implementazione e testing di modelli di fonti energetiche ambientali per reti di sensori senza fili autoalimentate} and by the PRIN 2015 Research Project {\it Variational methods, with applications to problems in mathematical physics and geometry}. The third author was supported by the ERC grant $\epsilon$ ({\it Elliptic Pde's and Symmetry of Interfaces and Layers for Odd Nonlinearities}). The authors were also supported by the SRA grants P1-0292, J1-7025, J1-6721, and J1-5435.}


\author[G. Molica Bisci]{Giovanni Molica Bisci}
\address{Dipartimento PAU,
          Universit\`a `Mediterranea' di Reggio Calabria,
          Via Melissari 24, 89124 Reggio Calabria, Italy}
\email{\tt gmolica@unirc.it}

\author[D. Repov\v{s}]{Du\v{s}an Repov\v{s}}
\address{Faculty of Education and Faculty of Mathematics and Physics,
         University of Ljubljana,
         1000 Ljubljana, Slovenia}
\email{dusan.repovs@guest.arnes.si}

\author[R. Servadei]{Raffaella Servadei}
\address{Dipartimento di Scienze Pure e Applicate (DiSPeA), Universit\`a degli Studi di Urbino
`Carlo Bo', Piazza della Repubblica 13, 61029 Urbino (Pesaro e Urbino), Italy}
\email{\tt raffaella.servadei@uniurb.it}



\begin{abstract}
This paper concerns with a class of elliptic equations on fractal domains depending on a real parameter. Our approach is based on variational methods. More precisely, the existence of at least two non-trivial weak (strong) solutions for the treated problem is obtained exploiting a local minimum theorem for differentiable functionals defined on reflexive Banach spaces. A special case of the main result improves a classical application of the Mountain Pass Theorem in the fractal setting, given by Falconer and Hu (1999).
\end{abstract}

\maketitle

\tableofcontents

\section{Introduction}
It is well-known that a great attention has been focused by many authors on the study of elliptic equations on fractal domains and in particular on the Sierpi\'nski gasket. See, among others, the papers \cite{Bre2,Bre1,BRaduV, BRV,Bre3,Bre4, Fa99} and \cite{FaHu, FuSc, ZHe,Hu, MR, stripaper0, stripaper}, as well as the references therein, where the authors obtained several existence and multiplicity results for problems on fractal domains under different growth assumptions on the data.

Motivated by this large interest in the literature, we study here the existence of weak (strong) solutions for the following parametric problem
\begin{equation}\label{Np0}
\left\{
\begin{array}{l}
\Delta u(x)+\alpha(x)u(x)=\lambda f(x,u(x))\quad x \in V\setminus V_0\\
u|_{V_0}=0,\\
\end{array}
\right.
\end{equation}
where $V$ stands for the Sierpi\'nski gasket in $(\RR^{N-1},|\cdot|)$, $N\geq 2$, $V_0$ is its intrinsic boundary (consisting of its $N$ corners), $\Delta$ denotes the weak Laplacian on $V$, $\lambda$ is a positive real parameter and $\alpha$ and $f$ are suitable functions.

The elliptic equation~\eqref{Np0} models some physical phenomena
such as reaction-diffusion problems, elastic properties of fractal media
and flow through fractal non-smooth domains and in all these cases the parameter $\lambda$ has a specific interpretation. When considering problems with parameters the interest is, on one hand, finding solutions, and, on the other hand, studying how these solutions depend on them.

A natural question is whether or not classical existence results for equation~\eqref{Np0} considered in bounded domains (see, for instance, \cite{ar,rabinowitz, struwe} and references therein) still hold in the fractal framework. Our contribution in this direction is stated in the following result:
\begin{theorem}\label{MolicaBisciPrincipal}
Let $\alpha\in L^1(V)$ be a function satisfying either
\begin{equation}\label{alfa1}
\alpha(x)\leq 0\,\,\, \mbox{for a.e.}\,\,\, x\in V
\end{equation}
or
\begin{equation}\label{alfa2}
\int_V|\alpha(x)|d\mu<\frac{1}{(2N+3)^2}
\end{equation}
and let
$f:V\times\RR\rightarrow\RR$ be a continuous function such that
\begin{equation}\label{f0}
f(x,0)\neq 0 \quad \mbox{for any}\,\, x\in V
\end{equation}
and
\begin{equation}\label{f2}
\begin{aligned}
& \qquad \mbox{there are}\,\, \nu>2\,\, \mbox{and}\,\, r_0>0\,\, \mbox{such that}\\
& tf(x,t)\leq \nu F(x,t)<0\,\, \mbox{for any}\,\, |t|\geq r_0,\,\, x\in V,
\end{aligned}
\end{equation}
where $F$ is the potential given by
\begin{equation}\label{F}
{\displaystyle F(x,t):=\int_0^t f(x,s)\,ds}\quad  \mbox{for any}\,\,(x,t)\in V\times\RR\,.
\end{equation}

Then, for any $\varrho>0$ and any
\begin{equation}\label{lambda}
0<\lambda<\frac{\varrho}{2\displaystyle\max_{{\footnotesize\begin{array}{c}
x\in V\\
|s|\leq \kappa\sqrt{\varrho}
\end{array}}}\left|\int_0^{s}f(x,t)dt\right|}\,,
\end{equation}
where
\begin{equation}\label{kappa}
\kappa:=\left\{\begin{array}{ll}
2N+3 & \mbox{if \eqref{alfa1} holds}\\
\\
{\displaystyle \frac{2N+3}{\sqrt{1-(2N+3)^2 {\displaystyle\int_V{|\alpha(x)|}d\mu}}}} & \mbox{if \eqref{alfa2} holds},
\end{array}\right.
\end{equation}
the problem~\eqref{Np0} admits at least two non-trivial weak solutions one of which lies in
$$
\mathbb{B}_\varrho:=\left\{u\in H_0^1(V):\|u\|_\alpha<\sqrt{\varrho}\right\}.
$$
\end{theorem}

Roughly speaking in Theorem~\ref{MolicaBisciPrincipal} we prove that, for small values of the parameter~$\lambda$, problem~\eqref{Np0} admits at least two non-trivial weak solutions, provided that the continuous and nonlinear term $f$ satisfies the celebrated Ambrosetti-Rabinowitz condition without any additional growth assumptions at infinity. A simple model for $f$ is given by the function
\begin{equation}\label{modellof}
f(x,t)=-a(x)\left(t^3+1\right)
\end{equation}
with $a\in C(V)$ and $a>0$ in $V$.

The proof of Theorem~\ref{MolicaBisciPrincipal} is based on variational techniques. This method is not trivial for consideration due to the fact that several difficulties which arise in the new geometrical context given by the Sierpi\'nski gasket have to be overcome. In particular, some analytical properties on the Hilbert space $H^1_0(V)$ need a special care (see Subsection~\ref{subsec:functionalspaces} for the details).
Also, we emphasize that the specific functional setting and techniques involved in handling fractal problems are different in comparison with those considered for classical Dirichlet problems.

Moreover, it is worth pointing out that the variational approach used to attack problems in fractal domains is not often easy to perform. For instance, in this setting there is no concept
of a derivative for a function, and so we need to clarify the notion
of Laplace operator on the fractal region: we would recall that this can be done explicitly only on some special fractals, such as, for instance, the Sierpi\'nski gasket $V$. Once a Laplacian is constructed on $V$, we can use
the Hilbert space $H^1_0(V)$ and its compactness properties in order to study our problem.

More precisely, the proof of Theorem~\ref{MolicaBisciPrincipal} relies on an abstract theorem proved in \cite{R0}, which is a joint application of the classical Pucci-Serrin Theorem (see \cite{puse}) and of a local minimum result obtained in \cite{R2} (see also \cite{An} for related topics).
As described in the forthcoming Subsection~\ref{subsec:variational} and Section~\ref{Section3}, our approach here is based on checking that the energy functional $\mathcal J_\lambda$ associated to problem~\eqref{Np0}
satisfies some geometrical conditions and the classical Palais-Smale property. Thanks to Proposition~\ref{PScondition}, in contrast with the standard elliptic case, the compactness condition for $\mathcal J_\lambda$ is satisfied without recourse to growth assumptions on the nonlinear term $f$.

Finally, it is interesting to note that in our approach the behavior of the nonlinearity $f$ at the origin is weaker than the one usually considered in the classical elliptic case and so Theorem~\ref{MolicaBisciPrincipal} improves the
paradigmatic application of the Mountain Pass Theorem for elliptic partial
differential equations on smooth domains given in \cite[Theorem~6.2]{struwe} (see also~\cite{ar,rabinowitz}). Theorem~\ref{MolicaBisciPrincipal} can be seen as the fractal counterpart of \cite[Theorem~4]{R0}, where the author studied the existence of solutions for an elliptic PDE, under growth conditions weaker than the usual ones (of superlinear and subcritical type).

A special case of Theorem~\ref{MolicaBisciPrincipal} reads as follows:
\begin{theorem}\label{MolicaBisciSpecial}
Let $\alpha\in C(V)$ satisfy \eqref{alfa1} and let
$f:V\times\RR\rightarrow\RR$ be a continuous function satisfying \eqref{f0}, \eqref{f2} and such that
\begin{equation}\label{f1}
\begin{aligned}
& \mbox{there are positive constants}\,\,\, M_0\,\, \mbox{and}\,\, \beta\,\, \mbox{such that}\\
& \displaystyle\max_{(x,s)\in V\times[-M_0,M_0]}|f(x,s)|\leq \frac{M_0}{2(\beta+1)(2N+3)^2}\,.
\end{aligned}
\end{equation}

Then, the following problem
\begin{equation}\label{Np}
\left\{
\begin{array}{l}
\Delta u(x)+\alpha(x)u(x)=f(x,u(x))\quad x \in V\setminus V_0\\
u|_{V_0}=0\\
\end{array}
\right.
\end{equation}
admits at least two strong non-trivial solutions one of which lies in $\mathbb{B}_{M_0^2/(2N+3)^2}$.
\end{theorem}

It is easily seen that a similar result can be obtained under the assumption \eqref{alfa2} in the weight $\alpha$.
We notice that Theorem~\ref{MolicaBisciSpecial} improves the conclusions of \cite[Theorem 3.5]{FaHu}, where, under hypotheses \eqref{f2} and \eqref{f1}, the authors proved just the existence of at least one (non-trivial) strong solution for problem~\eqref{Np} by employing the Mountain
Pass Theorem. Moreover, in the same result no explicit information about the localization of the solution is provided. Finally, we observe that the function $f$ given in \eqref{modellof} is a prototype for Theorem~\ref{MolicaBisciSpecial}.

The plan of the present paper is as follows. Section~\ref{section2} is devoted to the abstract framework and some preliminaries. In Section~\ref{sec:weakstrong} we give the notion of weak and strong solution for problem~\eqref{Np0} and its variational formulation. Later, in Section~\ref{Section3} we prove Theorem~\ref{MolicaBisciPrincipal} and Theorem~\ref{MolicaBisciSpecial} and we give some final comments.

\section{Abstract framework}\label{section2}

In this section we briefly recall some basic facts on the {Sierpi\'nski gasket} $V$ and the functional space $H_0^1(V)$ firstly introduced in \cite{FaHu} (see also \cite{Bre2, Bre1, BRaduV, BRV, Bre3, Bre4, MR}).

\subsection{The Sierpi\'nski gasket $V$}\label{subsec: sierpinski}
Let $N\geq2$ be a natural number and
let $p_1,\dots, p_N\in\RR^{N-1}$ be so that $|p_i-p_j|=1$ for $i\neq j$. Define, for every
$i\in\{1,\dots,N\}$, the map $S_i\colon\RR^{N-1}\to\RR^{N-1}$ by
$$S_i(x)=\frac12\,x+\frac12\,p_i\,.$$
 Let
${\mathcal S}:=\{S_1,\dots, S_N\}$ and denote by $L\colon{\mathcal
P}(\RR^{N-1})\to{\mathcal P}(\RR^{N-1})$ the map assigning to a subset
$A$ of $\RR^{N-1}$ the set
\begin{equation}\label{definizioneL}
L(A)=\bigcup_{i=1}^NS_i(A).
\end{equation}

It is well known that there
is a unique non-empty compact subset $V$ of $\RR^{N-1}$, called the
\emph{attractor} of the family ${\mathcal S}$, such that $L(V)=V$ (see \cite[Theorem~9.1]{Fa}). The set $V$ is called
the \emph{Sierpi\'nski gasket} in $\RR^{N-1}$ of \emph{intrinsic boundary} $V_0:=\{p_1,\dots,p_N\}$.\par

Let $\mu$ be the restriction to $V$ of the normalized $\log N/\log 2$-dimensional Hausdorff measure ${\mathcal
H}^d$ on $\RR^{N-1}$, so that $\mu(V)=1$ (see, for instance, Breckner, R\u adulescu and Varga \cite{BRaduV} for more details).
Finally, we also recall the following property of $\mu$ which will be useful in the sequel:
\begin{equation}\label{support}
\mu(B)>0,\hbox{ for every non-empty open subset $B$ of }V.
\end{equation}
\indent For a nice and interesting introduction to fractal geometry we refer to the monograph~\cite{Fa}.

\subsection{Functional spaces on $V$}\label{subsec:functionalspaces}
In what follows we denote by $C(V)$ the space of real-valued continuous functions on $V$ and by
$$C_0(V):=\Big\{u\in C(V): u|_{V_0}=0\Big\}.$$
The spaces $C(V)$ and $C_0(V)$ are endowed with the usual
norm $\|\cdot\|_\infty$.

For any function $u:V_*\to\RR$ and $m\in \NN$, let
\begin{equation}\label{defWm}
W_m(u):=\left(\frac{N+2}{N}\right)^m\sum_{\underset{|x-y|=2^{-m}}{x,y\in
V_m}}(u(x)-u(y))^2,
\end{equation}
where $V_m:=L(V_{m-1})$, $L$ is as in \eqref{definizioneL} and ${\displaystyle V_*:=\bigcup_{m\in \NN_0}V_m}$.

Since $W_m(u)\leq W_{m+1}(u)$ for any $m\in \NN$, we can construct the function $W(u)$ as follows
\begin{equation}\label{defW}
W(u):=\lim_{m\to\infty} W_m(u).
\end{equation}

Now, let $H_0^1(V)$ be the space given by
$$H_0^1(V):=\Big\{u\in C_0(V): W(u)<\infty\Big\}$$
equipped with the norm
\begin{equation}\label{norma}
\|u\|:=\sqrt{W(u)}.
\end{equation}

We conclude this subsection dealing with the compactness properties of $H_0^1(V)$.
With this respect, in our setting a key ingredient is given by the following Morrey-type inequality (see \cite[Lemma 2.4]{FaHu} for details)
\begin{equation}\label{si}
\displaystyle\sup_{x,y\in V_*}\frac{|u(x)-u(y)|}{|x-y|^{\sigma}}\leq (2N+3)\sqrt{W(u)},
\end{equation}
where
$$
\sigma:=\displaystyle\frac{\log((N+2)/N)}{2\log 2}\,.
$$
We would stress that the validity of inequality~\eqref{si} is due to the peculiar geometry of the Sierpi\'nski gasket $V$.

The Ascoli-Arz\'ela Theorem and \eqref{si} yield that the embedding
\begin{equation}\label{si2}
H_0^1(V)\hookrightarrow C_0(V)
\end{equation}
is compact (see \cite{FuSc}). Moreover, we get the following estimate:
\begin{equation}\label{stimau}
|u(x)|\leq (2N+3)\|u\|\,\,\, \mbox{for any}\,\, x\in V\,.
\end{equation}
Inequality~\eqref{stimau} will be crucial in the proof of Theorem~\ref{MolicaBisciPrincipal}, as we will see in the sequel (see Remark~\ref{notaMorrey}).

\section{Weak and strong solutions of problem~\eqref{Np0}}\label{sec:weakstrong}
In this section we give the notion of weak and strong solution for problem~\eqref{Np0} and we deal with its variational nature.
At this purpose we first give the notion of the Laplace operator on the Sierpi\'nski gasket $V$.

Following Falconer and Hu \cite{FaHu} we can define in a standard way a linear self-adjoint operator $\Delta: H^1_0(V)\to H^{-1}(V)$, where $H^{-1}(V)$ is the closure of
$L^2(V,\mu)$ with respect to the pre-norm
$$
\|u\|_{H^{-1}(V)}:=\sup_{\underset{\|h\|=1}{h\in H^1_0(V)}} |\langle u,h\rangle|,
$$
and
$$
\langle v,h\rangle=\int_Vv(x)h(x)d\mu\quad \mbox{for any}\,\,v\in L^2(V,\mu)\,\,\mbox{and}\,\, h\in H^1_0(V).
$$
\indent Note that $H^{-1}(V)$ is a Hilbert space. Then, the relation
$$
-{\mathcal W}(u,v)=\langle\Delta u,v\rangle\quad \mbox{for any}\,\, v\in H^1_0(V),
$$
where
 $$
 {\mathcal W}(u,v):=\lim_{m\rightarrow \infty}\left(\frac{N+2}{N}\right)^m\sum_{\underset{|x-y|=2^{-m}}{x,y\in V_m}}(u(x)-u(y))(v(x)-v(y))
 $$
 denotes the inner product in $H^1_0(V)$,
uniquely defines a function $\Delta u\in H^{-1}(V)$ for every $u\in
H^1_0(V)$. We call the operator $\Delta$ the \emph{weak Laplacian} on $V$.\par

Now, we can give the notion of weak solution for problem~\eqref{Np0}.
We say that a function $u\in H_0^1(V)$ is a \emph{weak solution} of problem~\eqref{Np0} if
$${\mathcal W}(u,v)-\int_V\alpha(x)u(x)v(x)d\mu+\lambda\int_V f(x,u(x))v(x)d\mu=0$$
for any $v\in H_0^1(V)$.

\begin{remark}\label{strong}
\rm{Note that if the standard Laplacian of a function $u\in H^1_0(V)$ exists, then this implies
the existence of the weak Laplacian of $u$ (see, for the sake of completeness, the paper \cite{FaHu}).
Furthermore, if $f$ and $\alpha$ are continuous, then \cite[Lemma~2.16]{FaHu} yields that every weak solution of problem~\eqref{Np0} is also a \textit{strong} solution of it (see \cite[Section 2 and Proposition 2.12]{FaHu}).}
\end{remark}

\subsection{Variational framework of the problem}\label{subsec:variational}
Problem~\eqref{Np0} is of variational nature, indeed the natural energy functional associated with it is given by $\mathcal J_\lambda:H^1_0(V)\rightarrow\RR$ defined as
\begin{equation}\label{Fu2}
\begin{aligned}
\mathcal J_\lambda(u)& :=\frac{W(u)}{2\lambda}-\frac{1}{2\lambda}\int_V {\alpha(x)}|u(x)|^{2}\,d\mu
+\int_V F(x, u(x))d\mu\,.
\end{aligned}
\end{equation}
\indent Note that the functional $\mathcal{J}_{\lambda}$ is continuously G\^{a}teaux differentiable at $u\in H_0^1(V)$ and one has
$$
\langle \mathcal{J}'_{\lambda}(u), v\rangle = \frac{{\mathcal W}(u,v)}{\lambda}-\frac{1}{\lambda}\int_V\alpha(x)u(x)v(x)d\mu
+\int_V f(x,u(x))v(x)d\mu
$$
for any $v \in H^1_0(V)$, thanks to \cite[Proposition~2.19]{FaHu}.\par

 Thus, the critical points of $\mathcal{J}_{\lambda}$ are exactly the weak solutions of problem~\eqref{Np0}.

We recall that a $C^1$-functional $J:E\to\RR$, where $E$ is a real Banach
space with topological dual $E^*$, satisfies the \emph{Palais-Smale condition} (in short $(\rm PS)$-condition) when
$$\begin{aligned}
& \quad \emph{every sequence $\{u_j\}_{j\in \NN}$ in $E$ such that
$\{J(u_j)\}_{j\in\NN}$ is bounded and}\\
& \emph{$\|J'(u_j)\|_{E^*}\to 0$ as $j\rightarrow +\infty$
possesses a convergent subsequence in $E$.}
\end{aligned}$$

The abstract tool used along the present paper in order to prove the existence of weak solutions for \eqref{Np0} is the following theorem (see \cite[Theorem~6]{R0}):
\begin{theorem}\label{Pucci-Serrin+Ricceri}
Let $E$ be a reflexive real Banach space and let $\Phi,\Psi:E\to\RR$
be two continuously G\^{a}teaux differentiable functionals such that
\begin{itemize}
\item $\Phi$ is
sequentially weakly lower semicontinuous and coercive in $E$
\item $\Psi$
is sequentially weakly continuous in $E$.
\end{itemize}
In addition, assume that for each $\mu>0$ the functional $J_\mu:=\mu\Phi-\Psi$ satisfies the $(\rm PS)$-condition. Then, for each $\varrho>\displaystyle\inf_E\Phi$ and each
$$\mu>\inf_{u\in\Phi^{-1}\big((-\infty,\varrho)\big)}
\frac{\displaystyle\sup_{v\in\Phi^{-1}\big((-\infty,\varrho)\big)}\Psi(v)-\Psi(u)}{\varrho-\Phi(u)}\,\,,$$
the following alternative holds: either the functional $J_\mu$ has a strict global minimum which lies in $\Phi^{-1}\big((-\infty,\varrho))$, or $J_\mu$ has at least two critical points one of which lies in $\Phi^{-1}\big((-\infty,\varrho))$.
\end{theorem}

Theorem~\ref{Pucci-Serrin+Ricceri} comes out from a joint application of the classical Pucci-Serrin Theorem (see \cite{puse}) and a local minimum result due to Ricceri (see \cite{R2}). We refer the interested reader to \cite{BoMoRa0,BoMoRa1,MR,Ri1,R3} and references therein for some applications of Ricceri's variational principle and to \cite{k2} for related topics on the variational methods used in this paper (see also the classical reference~\cite{brezis}).

The $(\rm PS)$-condition is one of the main compactness assumption required on the energy functional when considering critical point theorem. In order to simplify its proof, in the sequel we will perform the following result, which is valid for the functional~$\mathcal J_\lambda$ given in \eqref{Fu2}:
\begin{proposition}\label{PScondition}
Let $f\in C(V\times\RR)$ and $\alpha\in L^1(V)$ and let $\mathcal J_\lambda$ be the energy functional
defined in \eqref{Fu2}.
If the sequence $\{u_j\}_{j\in \NN}$ is bounded in $H^1_0(V)$ and
$$
\|\mathcal{J}_{\lambda}'(u_j)\|_{H^{-1}(V)}\to0\quad \mbox{as}\,\,j\rightarrow +\infty\,,
$$
then $\{u_j\}_{j\in \NN}$ has a Cauchy subsequence in $H^1_0(V)$ and so $\{u_j\}_{j\in \NN}$ has a convergent subsequence.
\end{proposition}
\begin{proof}
See \cite[Proposition~2.24]{FaHu} for a detailed proof.
\end{proof}

\section{Main results of the paper}\label{Section3}
The aim of this section is to prove that, under natural assumptions on the nonlinear term~$f$, problem~\eqref{Np0} admits two non-trivial solutions. As we already said, this is done by means of variational techniques.

Before proving Theorem~\ref{MolicaBisciPrincipal} it will be useful to define another norm on $H^1_0(V)$
as follows:
\begin{equation}\label{eqnorm}
\|u\|_{\alpha}:=\sqrt{\displaystyle W(u)-\int_V{\alpha(x)}|u(x)|^{2}d\mu},
\end{equation}
where $\alpha$ is the function satisfying the assumptions stated in Theorem~\ref{MolicaBisciPrincipal} and $W$ is defined in \eqref{defW}.
It is easy to see that $\|\cdot\|_\alpha$ is a norm on $H^1_0(V)$ equivalent to the usual one given in \eqref{norma}.

Indeed, if $\alpha$ satisfies condition~\eqref{alfa1} we have that
\begin{equation}\label{normaalfa1}
\begin{aligned}
\|u\|_{\alpha}^2=\displaystyle W(u)-\int_V{\alpha(x)}|u(x)|^{2}d\mu\geq W(u)=\|u\|^2,
\end{aligned}
\end{equation}
and, by \eqref{stimau}, we get
$$\begin{aligned}
\|u\|_{\alpha}^2 & =\displaystyle W(u)-\int_V{\alpha(x)}|u(x)|^{2}d\mu\\
& \leq \displaystyle W(u)-(2N+3)^2\|u\|^2\int_V{\alpha(x)}d\mu\\
& = \Big(1+(2N+3)^2 \int_V{|\alpha(x)|}d\mu\Big)\|u\|^2\,.
\end{aligned}$$

On the other hand, if $\alpha$ verifies condition~\eqref{alfa2} we have that
\begin{equation}\label{normaalfa2}
\begin{aligned}
\|u\|_{\alpha}^2 & =\displaystyle W(u)-\int_V{\alpha(x)}|u(x)|^{2}d\mu\\
& \geq \displaystyle W(u)-\int_V{|\alpha(x)|}|u(x)|^{2}d\mu\\
& \geq \displaystyle W(u)-(2N+3)^2\|u\|^2\int_V|\alpha(x)|d\mu\\
& = \Big(1-(2N+3)^2 \int_V{|\alpha(x)|}d\mu\Big)\|u\|^2
\end{aligned}
\end{equation}
and
$$\begin{aligned}
\|u\|_{\alpha}^2 & =\displaystyle W(u)-\int_V{\alpha(x)}|u(x)|^{2}d\mu\\
& \leq \displaystyle W(u)+\int_V{|\alpha(x)|}|u(x)|^{2}d\mu\\
& \leq \displaystyle W(u)+(2N+3)^2\|u\|^2\int_V{|\alpha(x)|}d\mu\\
& \leq 2\|u\|^2\,,
\end{aligned}$$
thanks to \eqref{stimau}.

Now we can prove our main results.
\subsection{Proof of Theorem~\ref{MolicaBisciPrincipal}}\label{subsec:proofthP}
The idea of the proof consists in applying Theorem~\ref{Pucci-Serrin+Ricceri} to the functional
$$\mathcal J_\lambda(u)=\frac{1}{2\lambda}\Phi(u)-\Psi(u)\,,$$
where
$$
\Phi(u):=\|u\|_{\alpha}^2,
$$
as well as
$$
\Psi(u):=-\displaystyle\int_V F(x,u(x))d\mu,
$$
for any $u\in H^1_0(V)$. Note that here we perform Theorem~\ref{Pucci-Serrin+Ricceri} taking the parameter ${\displaystyle\mu=\frac{1}{2\lambda}}$.

First of all, let us consider the regularity assumptions required on $\Phi$ and $\Psi$. It is easy to see that $\Phi$ is sequentially weakly lower semicontinuous and coercive in $H^1_0(V)$.

Now, let us prove that $\Psi$ is sequentially weakly continuous in $H^1_0(V)$. At this purpose, let $\{u_j\}_{j\in \NN}$ be a sequence in $H^1_0(V)$ such that
$$u_j \to u\quad \mbox{weakly in}\,\, H^1_0(V)$$
as $j\to +\infty$, for some $u\in H^1_0(V)$. Then, by \eqref{si2}, we get that
$$u_j \to u\quad \mbox{in}\,\, C_0(V)\,,$$
that is
\begin{equation}\label{normalinfty}
\|u_j - u\|_\infty\to 0
\end{equation}
as $j\to +\infty$.
As a consequence of \eqref{normalinfty} we get that there exists a positive constant $K$ such that
\begin{equation}\label{norma<K}
\|u_j\|_\infty\leq K \quad \mbox{and} \quad \|u\|_\infty\leq K\,\,\, \mbox{for any}\,\, j\in \NN\,.
\end{equation}

Hence, we deduce that
\begin{equation}\label{psiweakly}
\begin{aligned}
\Big|\Psi(u_j)-\Psi(u)\Big| & =\Big|\int_V F(x, u_j(x))d\mu- \int_V F(x, u(x))d\mu\Big|\\
& \leq  \int_V \Big|F(x, u_j(x))- F(x, u(x))\Big|d\mu\\
& = \int_V \Big|\int_{u_j(x)}^{u(x)}f(x,t)dt\Big|d\mu\\
& \leq \int_V \Big|\int_{u_j(x)}^{u(x)}|f(x,t)|dt\Big|d\mu\\
& \leq \int_V \Big|\int_{u_j(x)}^{u(x)}{\displaystyle \max_{|t|\leq K}|f(x,t)|}dt\Big|d\mu\\
& = \int_V {\displaystyle \max_{|t|\leq K}|f(x,t)|}\,\,|u_j(x) - u(x)| d\mu\\
& \leq {\displaystyle \max_{x\in V,\, |t|\leq K}|f(x,t)|}\,\,\|u_j - u\|_\infty\,,
\end{aligned}
\end{equation}
since \eqref{norma<K} holds true, $f$ is continuous in $V\times \mathbb R$ and $V$ is compact with $\mu(V)=1$\,. By \eqref{normalinfty} and \eqref{psiweakly} we obtain that
$$\Big|\Psi(u_j)-\Psi(u)\Big|\to 0$$
as $j\to +\infty$, so that $\Psi$ is sequentially weakly continuous in $H^1_0(V)$.

Now, we observe that
\begin{equation}\label{junbounded}
\mbox{the functional $\mathcal J_\lambda$ is unbounded from below in $H^1_0(V)$\,.}
\end{equation}
Indeed, assumption~\eqref{f2} implies that there exist two positive constants $b_1$ and $b_2$ such that
\begin{equation}\label{inequal}
F(x,t)\leq -b_1|t|^{\nu}+b_2\quad \mbox{for any}\,\, x\in V\, \mbox{and}\,\, t\in \RR.
\end{equation}
Thus, by \eqref{inequal} and the fact that $\mu (V)=1$,  for any $u\in H^1_0(V)$ one has
\begin{equation}\label{inequal2}
\int_VF(x,u(x))d\mu\leq -b_1\int_V|u(x)|^{\nu}d\mu+b_2\,.
\end{equation}
Let $u_0\in H^1_0(V)$ with $\|u_0\|_\alpha=1$. Then, by \eqref{inequal2} we have that
$$\begin{aligned}
\mathcal J_\lambda(t u_0) &= \frac{t^2}{2\lambda}+\int_V F(x,tu_0(x))d\mu\\
&\leq \frac{t^2}{2\lambda}-b_1|t|^{\nu}\int_V|u_0(x)|^{\nu}d\mu+b_2\\
&\rightarrow -\infty,
\end{aligned}$$
as $t\rightarrow +\infty$, since $\nu>2$ by assumption~\eqref{f2} and ${\displaystyle \int_V|u_0(x)|^{\nu}d\mu>0}$\,. This concludes the proof of \eqref{junbounded}.

Now, it remains to prove that the functional $\mathcal J_\lambda$ verifies the $(\rm PS)$-condition. To this goal, it is enough to argue as in \cite[Theorem 3.5]{FaHu} and to use Proposition~\ref{PScondition}.

Finally, let $\varrho>0$ and
$$\chi(\varrho):=\displaystyle\inf_{u\in \mathbb{B}_{\varrho}}\frac{\displaystyle\sup_{v\in \mathbb{B}_{\varrho}}\Psi(v) -\Psi(u)}{\varrho-\|u\|_{\alpha}^2}\,,$$
where
$$\mathbb{B}_{\varrho}=\Big\{v\in H^1_0(V) : \|v\|_\alpha<\sqrt \varrho\Big\}\,.$$

The definition of $\chi$ yields that for every $u\in \mathbb{B}_{\varrho}$
$$\chi(\varrho) \leq \frac{\displaystyle\sup_{v\in \mathbb{B}_{\varrho}}\Psi(v) -\Psi(u)}{\varrho-\|u\|_{\alpha}^2}$$
so that, using the fact that $0\in \mathbb{B}_\varrho$, we obtain that
\begin{equation}\label{Funzionale1}
\begin{aligned}
\chi(\varrho)
& \leq \frac 1 \varrho\displaystyle\sup_{v\in \mathbb{B}_{\varrho}}\Psi(v)\\
& \leq \frac 1 \varrho \displaystyle\sup_{v\in \overline{\mathbb{B}}_\varrho}\left|\int_V F(x,v(x))d\mu\right|\\
& \leq \frac 1 \varrho \displaystyle\sup_{v\in \overline{\mathbb{B}}_\varrho}\int_V \left|F(x,v(x))\right| d\mu.
\end{aligned}
\end{equation}

Now, assume that the function $\alpha$ satisfies assumption~\eqref{alfa1}. Then, if $v\in \overline{\mathbb{B}}_\varrho$, by \eqref{stimau} and \eqref{normaalfa1} we get that
\begin{equation}\label{valfa1}
|v(x)|\leq (2N+3)\|v\|\leq (2N+3)\|v\|_\alpha\leq (2N+3)\sqrt{\varrho} \quad \mbox{for any}\,\, x\in V\,,
\end{equation}
which combined with the continuity of $F$ and the compactness of $V$ gives for any $x\in V$
\begin{equation}\label{Funzionale2F}
\left|F(x,v(x))\right|\leq \displaystyle\max_{{\footnotesize\begin{array}{c}
y\in V\\
|s|\leq (2N+3)\sqrt{\varrho}
\end{array}}}\left|\int_0^{s}f(y,t)dt\right|.
\end{equation}
Therefore, bearing in mind that $\mu(V)=1$, inequality~\eqref{Funzionale2F} yields
\begin{equation}\label{Funzionale2F2}
\int_V \left|F(x,v(x))\right|d\mu\leq \displaystyle\max_{{\footnotesize\begin{array}{c}
x\in V\\
|s|\leq (2N+3)\sqrt{\varrho}
\end{array}}}\left|\int_0^{s}f(x,t)dt\right|
\end{equation}
for any $v\in \overline{\mathbb{B}}_{\varrho}$.

By \eqref{Funzionale1} and \eqref{Funzionale2F2} we have that
$$\chi(\varrho)\leq \frac 1 \varrho \displaystyle\max_{{\footnotesize\begin{array}{c}
x\in V\\
|s|\leq (2N+3)\sqrt{\varrho}
\end{array}}}\left|\int_0^{s}f(x,t)dt\right|<\frac{1}{2\lambda}\,,$$
provided $\lambda$ satisfies condition \eqref{lambda}.

If the function $\alpha$ satisfies assumption~\eqref{alfa2}, we can argue in the same way, just replacing \eqref{valfa1} with the following inequality
\begin{equation}\label{valfa2}
\begin{aligned}
|v(x)| & \leq (2N+3)\|v\|\\
& \leq \frac{2N+3}{\sqrt{1-(2N+3)^2 {\displaystyle \int_V{|\alpha(x)|}d\mu}}}\|v\|_\alpha\\
& \leq \frac{2N+3}{\sqrt{1-(2N+3)^2 {\displaystyle \int_V{|\alpha(x)|}d\mu}}}\sqrt{\varrho}
\end{aligned}
\end{equation}
for any $x\in V$, thanks to \eqref{normaalfa2}.

In both cases, owing to Theorem~\ref{Pucci-Serrin+Ricceri} and taking into account \eqref{f0} and \eqref{junbounded}, we conclude that problem~\eqref{Np0}
admits at least two non-trivial weak solutions one of which lies in $\mathbb{B}_\varrho$.
The proof of Theorem~\ref{MolicaBisciPrincipal} is now complete.

\medskip

In order to conclude this subsection, in the sequel we remark some facts.
\begin{remark}\label{notaMorrey}
{\rm First of all, we notice that condition~\eqref{si2} plays a crucial role in the proof of Theorem~\ref{MolicaBisciPrincipal}, whereas the Sobolev embedding theorems are employed in the classical case of bounded domains (see, among others, \cite{ar, rabinowitz, struwe}).

Moreover, we would stress that the maximal interval of $\lambda$'s where the conclusion of Theorem~\ref{MolicaBisciPrincipal} holds true is given by $(0, \lambda^*)$, where
$$
\lambda^*  :=\frac 1 2 \sup_{\varrho>0}\frac{\varrho}{\displaystyle\max_{{\footnotesize\begin{array}{c}
x\in V\\
|s|\leq \kappa\sqrt{\varrho}
\end{array}}}\left|\int_0^{s}f(x,t)dt\right|}= \frac{1}{2\kappa^2}\sup_{z>0}\frac{z^2}{\displaystyle\max_{{\footnotesize\begin{array}{c}
x\in V\\
|s|\leq z
\end{array}}}\left|\int_0^{s}f(x,t)dt\right|}\,,
$$
with $\kappa$ as in \eqref{kappa}.

Finally, note that if we require that $\alpha\in C(V)$, we get the existence of two non-trivial strong solutions for problem~\eqref{Np0} by Remark~\ref{strong}.}
\end{remark}

\begin{remark}\label{corollarioprova2}
\rm{Note that the trivial function
is a weak solution of problem~\eqref{Np0} if and only if $f(\cdot, 0)= 0$. Hence, condition~\eqref{f0} assures that all the solutions of problem~\eqref{Np0}, if any, are non-trivial.

In the case when $f(\cdot,0) = 0$, in order to get the existence of a non-trivial solution for \eqref{Np0} (and so a multiplicity result) we need some extra assumptions on the nonlinear term $f$.  For instance, in \cite{FeMoRe} the authors assumed the following subquadratical growth condition at zero
$$
\liminf_{t\rightarrow 0^+}\frac{\displaystyle F(x,t)}{t^2}=-\infty\,\,\, \mbox{uniformly in $V$,}
$$
in addition to \eqref{f2}.}
\end{remark}

\subsection{Proof of Theorem~\ref{MolicaBisciSpecial}}\label{subsec:proofthP}
Theorem~\ref{MolicaBisciSpecial} is a direct consequence of Theorem~\ref{MolicaBisciPrincipal}. Indeed,
as it is easily seen, condition~\eqref{f1} yields
\begin{equation}\label{Funzionale2F3}
\displaystyle\max_{(x,s)\in V\times[-M_0,M_0]}\left|\int_0^{s}f(x,t)dt\right|\leq \frac{M_0^2}{2(\beta+1)(2N+3)^2}.
\end{equation}

Thus, by the fact that $\beta>0$ and \eqref{Funzionale2F3} holds, we get that
$$
1<\beta+1\leq \frac{M_0^2}{2(2N+3)^2\displaystyle\max_{(x,s)\in V\times[-M_0,M_0]}\left|\int_0^{s}f(x,t)dt\right|}\leq \lambda^*\,,
$$
thanks to Remark~\ref{notaMorrey} and the fact that $\alpha$ satisfies \eqref{alfa1}.

Then, applying Theorem~\ref{MolicaBisciPrincipal} with $\lambda=1$ we obtain that problem~\eqref{Np}
admits at least two non-trivial weak solutions one of which lies in $\mathbb{B}_{M_0^2/(2N+3)^2}$. Finally, by the regularity assumptions on the nonlinear term $f$ and the weight $\alpha$, Remark~\ref{strong} ensures that every weak solution of problem~\eqref{Np} is also strong and this concludes the proof of Theorem~\ref{MolicaBisciSpecial}.

\end{document}